\documentclass[a4paper, 11pt]{amsart}

\title[symmetric products and SODs]
{symmetric products of dg categories and semi-orthogonal decompositions}
\author{Naoki Koseki}
\date{}

\address{The University of Liverpool, Mathematical Sciences Building, Liverpool, L69 7ZL, UK.}
\email{koseki@liverpool.ac.uk}

\makeatletter
 
  \@addtoreset{equation}{section}
 \makeatother

\usepackage{amsmath, amssymb, amsthm, amscd, comment, mathtools, color}
\usepackage{todonotes}
\usepackage[frame,cmtip,curve,arrow,matrix,line,graph]{xy}
\usepackage[mathscr]{euscript}
\usepackage{mathrsfs}
\usepackage{tikz}
\usetikzlibrary{intersections, calc}

\theoremstyle{plain}
\newtheorem{thm}{Theorem}[section]
\newtheorem{prop}[thm]{Proposition}
\newtheorem{def-prop}[thm]{Definition-Proposition}
\newtheorem{lem}[thm]{Lemma}
\newtheorem{cor}[thm]{Corollary}
\newtheorem*{thm*}{Theorem}

\theoremstyle{definition}
\newtheorem{defin}[thm]{Definition}

\newtheorem*{NaC}{Notation and Convention}
\newtheorem*{ACK}{Acknowledgement}

\theoremstyle{remark}
\newtheorem{rmk}[thm]{Remark}

\newtheorem{ex}[thm]{Example}

\DeclareMathOperator{\id}{id}

\newcommand{\bP}{\mathbb{P}}
\newcommand{\bC}{\mathbb{C}}

\newcommand{\bZ}{\mathbb{Z}}

\newcommand{\bK}{\mathbb{K}}
\newcommand{\mcA}{\mathcal{A}}
\newcommand{\mcB}{\mathcal{B}}
\newcommand{\mcC}{\mathcal{C}}
\newcommand{\mcD}{\mathcal{D}}

\newcommand{\mcO}{\mathcal{O}}

\newcommand{\mcT}{\mathcal{T}}

\DeclareMathOperator{\Hilb}{Hilb}
\DeclareMathOperator{\Sym}{Sym}

\DeclareMathOperator{\pt}{pt}

\DeclareMathOperator{\pr}{pr}
\DeclareMathOperator{\op}{op}

\DeclareMathOperator{\Ind}{Ind}
\DeclareMathOperator{\Res}{Res}
\DeclareMathOperator{\pre-tr}{pre-tr}

\DeclareMathOperator{\HH}{HH}
\DeclareMathOperator{\Perf}{Perf}

\DeclareMathOperator{\Mod}{Mod}

\newcommand{\hatS}{\hat{S}}

\newcommand{\fS}{{\mathfrak S}}

\begin{document}
\maketitle

\begin{abstract}
In this article, we investigate semi-orthogonal decompositions of 
the symmetric products of dg-enhanced triangulated categories. 
Given a semi-orthogonal decomposition $\mcD=\langle \mcA, \mcB \rangle$, 
we construct semi-orthogonal decompositions of the symmetric products of $\mcD$ 
in terms of that of $\mcA$ and $\mcB$. 
This was originally stated  by Galkin--Shinder, and 
answers the question raised by Ganter--Kapranov. 

Combining the above result with the derived McKay correspondence, 
we obtain various interesting semi-orthogonal decompositions 
of the derived categories of the Hilbert schemes of points on surfaces. 
\end{abstract}

\setcounter{tocdepth}{1}
\tableofcontents

\section{Introduction}
\subsection{Motivation and Results}
For a given (dg-enhanced) triangulated category $\mcD$, 
Ganter--Kapranov \cite{gk14} defined its symmetric products $\Sym^n\mcD$ 
for all non-negative integer $n \geq 0$. 
This notion of symmetric products of dg categories provides interesting subjects 
in representation theory. Namely, Gyenge--Koppensteiner--Logvinenko \cite{gkl21} 
constructed a {\it categorical Heisenberg action} on 
$\oplus_{n \geq 0} \Sym^n\mcD$. 
Their result generalizes and unifies various previous results in the literature, e.g., 
\cite{cl12, gro96, kru18, nak97}. 

When $\mcD$ is 
the bounded derived category $D^b(S)$ of coherent sheaves 
on a smooth projective surface $S$, we have 
\begin{equation} \label{eq:introSymHilb}
\Sym^nD^b(S) \cong D^b([S^{\times n}/\fS_n]) \cong D^b(\Hilb^n(S)), 
\end{equation}
where the second equivalence is highly non-trivial 
and follows from the derived McKay correspondence \cite{bkr01, hai01}. 
In this special case, the construction of \cite{gkl21} categorifies the famous Heisenberg action 
on the cohomology groups of the Hilbert schemes of points due to 
Grojnowksi \cite{gro96} and Nakajima \cite{nak97}. 

The goal of the present paper is to understand 
the interaction between the notion of symmetric products and 
one of the most fundamental concepts in the dg category theory; 
{\it semi-orthogonal decompositions}. 
Semi-orthogonal decompositions of dg categories appear in various settings, 
especially in algebraic geometry 
\cite{bo95, kaw02, kuz14}. 

The following is the main theorem in this paper: 
\begin{thm}[Theorem \ref{thm:SymSOD}] \label{thm:intromain}
Let $\mcD$ be a dg-enhanced triangulated category. 
Suppose that we have a semi-orthogonal decomposition 
$\mcD=\langle \mcA, \mcB \rangle$. 

Then we have a semi-orthogonal decomposition 
\[
\Sym^n\mcD=\left\langle
\Sym^{n-i} \mcA \bullet \Sym^i\mcB \colon i=0, \cdots, n
\right\rangle
\]
for each $n \in \bZ_{>0}$. 
\end{thm}

\begin{rmk}
The above semi-orthogonal decomposition first appeared in \cite[Equation (3)]{gs15} 
without a proof. 
\end{rmk}

In particular, we have: 
\begin{cor}[Corollary \ref{cor:fullexcep}]
\label{cor:intro}
Let $\mcD$ be a dg-enhanced triangulated category. Suppose that $\mcD$ has a full exceptional collection of length $l$. 
Then $\Sym^n\mcD$ has a full exceptional collection of length 
\begin{equation*} 
q(n; l) \coloneqq \sum_{\substack{i_1+\cdots+i_l=n \\ i_1, \cdots, i_l \geq 0}}
p(i_1) \cdot \cdots \cdot p(i_l). 
\end{equation*}
\end{cor}

In the above theorem, $(-) \bullet (-)$ is a version of a tensor product for dg categories 
introduced by \cite{bll04}, see Section \ref{sec:bullet} for more detail. 
We can think of this result as a natural generalization of 
the following direct sum decomposition for vector spaces $V, W$: 
\[
\Sym^n(V \oplus W)=\oplus_{i=0}^n \Sym^{n-i}V \otimes \Sym^iW. 
\]
This would be a satisfying answer to the question raised by Ganter--Kapranov 
\cite[Question 7.2.1]{gk14}. 
The key ingredient of the proof is 
Elagin's descent theory for semi-orthogonal decompositions, see \cite{ela12, shi18}. 

The most interesting case for us is the case of $\mcD=D^b(S)$, 
where $S$ is a smooth projective surface. 
Combined with the isomorphisms (\ref{eq:introSymHilb}), 
Theorem \ref{thm:intromain} provides a useful way 
to construct various interesting semi-orthogonal decompositions on $D^b(\Hilb^n(S))$. 
In Section \ref{sec:Hilb}, we consider the following cases: 
\begin{enumerate}
\item When $S=\bP^2$ or more generally a toric surface. 
In this case, $D^b(S)$ has a full exceptional collection. 
By Corollary \ref{cor:intro}, 
the same holds for 
$D^b(\Hilb^n(S))$ for all $n \geq 1$. 

\item When $D^b(S)$ has a semi-orthogonal collection consisting of exceptional objects and a category $\mcA$ whose Hochschild homology vanishes. 
Such a category $\mcA$ is called a {\it quasi-phantom}. 
See e.g. \cite{ao13, bgks15, bgs13, kkl17, gkms15, gs13} 
for examples of such surfaces. 
Using Theorem \ref{thm:intromain}, we will show that 
$\Sym^i\mcA \subset D^b(\Hilb^n(S))$ are phantom subcategories for all 
$1 \leq i \leq n$. 

\item When $S \to C$ is a $\bP^1$-bundle over a smooth projective curve $C$. 
In this case, we have $D^b(S)=\langle D^b(C), D^b(C) \rangle$. 
Combining Theorem \ref{thm:intromain} with some other results \cite{pvdb19, tod21d}, 
we prove that $D^b(\Hilb^n(S))$ has a semi-orthogonal decomposition 
whose components are derived categories of the products of 
the Jacobian $J(C)$ and $\Sym^iC$ for $0 \leq i \leq \min\{n, g(C)-1\}$. 

\item When $\hatS$ is the blow up of $S$ at a point. 
Then we have a semi-orthogonal decomposition 
$D^b(\hatS)=\langle D^b(S), D^b(\pt) \rangle$. 
By Theorem \ref{thm:intromain}, $D^b(\Hilb^n(\hatS))$ has a semi-orthogonal decomposition 
in terms of $D^b(\Hilb^i(S))$ for $0 \leq i \leq n$. 
\end{enumerate}

\subsection{Relation with existing works}
\begin{enumerate}
\item As mentioned above, Theorem \ref{thm:intromain} already appeared in 
\cite{gs15} without a proof. 
See also \cite[Remark 4.7]{ks15} for a special case but without using the notion 
of symmetric products of dg categories. 

However, the author could not find rigorous proof in the existing works, 
so decided to write the present paper. 

\item Semi-orthogonal decompositions of $D^b(\Hilb^n(S))$ 
have already been constructed for several algebraic surfaces $S$. 
For example, the semi-orthogonal decomposition for the derived category 
$D^b(\Hilb^n(\hatS))$, 
where $\hatS$ is the blow-up of a smooth projective surface at a point, 
was constructed in the author's recent paper \cite{kos21b} via a completely different method. 

The fact that $D^b(\Hilb^n(\bP^2))$ has a full-exceptional collection 
is obtained in \cite[Proposition 1.3]{ks15}. 
\end{enumerate}

\subsection{Plan of the paper}
In Section \ref{sec:prelim}, we recall some basic facts about 
dg categories and equivariant categories. 
In Section \ref{sec:main}, we prove Theorem \ref{thm:intromain}. 
In Section \ref{sec:Hilb}, we treat various examples of symmetric products of the derived categories on smooth projective varieties.

\begin{ACK}
The author would like to thank Professors Arend Bayer and Yukinobu Toda 
for fruitful discussions. 
He would also like to thank Professors Andreas Krug and Evgeny Shinder for helpful conversations. 
This work was supported by ERC Consolidator grant WallCrossAG, no.~819864.

Finally, the author would like to thank referees for their careful reading of this paper and pointing out several errors in the previous version. 
\end{ACK}

\begin{NaC}
Throughout the paper, we work over a field $\bK$. 
We use the following notations: 
\begin{itemize}
\item For a smooth projective variety $X$, $D^b(X)$ denotes 
the bounded derived category of coherent sheaves on $X$. 
\item For a dg category $\mcD$, we denote its homotopy category by $H^0(\mcD)$. 
\end{itemize}
\end{NaC}

\section{Preliminaries} \label{sec:prelim}
\subsection{Dg categories} \label{sec:dg}
In this section, we fix notation from dg category theory. 
For details, we refer to the papers \cite{bk91, kel06}. 

\subsubsection{Pre-triangulated dg categories}
Given a dg category $\mcD$, there is a canonically defined dg category 
$\mcD^{\pre-tr}$ with an embedding $\mcD \hookrightarrow \mcD^{\pre-tr}$ 
such that the homotopy category $H^0(\mcD^{\pre-tr})$ is a triangulated category \cite{bk91}. 
The category $\mcD^{\pre-tr}$ is called the {\it pre-triangulated hull} of $\mcD$. 
We say that a dg category $\mcD$ is {\it pre-triangulated} if the induced functor 
$H^0(\mcD) \hookrightarrow H^0(\mcD^{\pre-tr})$ is an equivalence. 

A {\it dg enhancement} of a triangulated category $\mcT$ is a pair of 
pre-triangulated dg category $\mcD$ and an equivalence $H^0(\mcD) \cong \mcT$. 
If a triangulated category $\mcT$ has a dg enhancement, 
we call $\mcT$ a {\it dg-enhanced} triangulated category. 

For a smooth projective variety $X$, 
its derived category $D^b(X)$ has a standard dg enhancement $I(X)$, 
the dg category of complexes of bounded below injective $\mcO_X$-modules with bounded coherent cohomology.

\subsubsection{$\bullet$-product} \label{sec:bullet}
For a dg category $\mcD$, 
$\mcD^{\op}\mathchar`-\Mod$ denotes the dg category of right $\mcD$-modules, 
and $\Perf\mathchar`-\mcD$ denotes the dg category of perfect modules. 
We have the following embeddings: 
\[
\mcD \subset \Perf\mathchar`-\mcD \subset \mcD^{\op}\mathchar`-\Mod. 
\]
Note that $\Perf\mathchar`-\mcD$ is pre-triangulated. 

\begin{defin}[{\cite[Definition 4.2]{bll04}}]
For dg categories $\mcA$ and $\mcB$, we define the dg category $\mcA \bullet \mcB$ as 
\[
\mcA \bullet \mcB \coloneqq \Perf\mathchar`-(\mcA \otimes \mcB). 
\]
\end{defin}

The following is the case we are mostly interested in: 
\begin{thm}[{\cite[Theorem 5.5]{bll04}}]
Let $X, Y$ be smooth projective varieties. 
Let $I(X), I(Y)$ be 
the standard dg enhancements of $D^b(X), D^b(Y)$, respectively. 
Then we have an equivalence 
\[
H^0(I(X) \bullet I(Y)) \cong D^b(X \times Y). 
\]
\end{thm}

The $\bullet$-product is well-behaved under semi-orthogonal decompositions: 
\begin{prop}[{\cite[Proposition 4.6]{bll04}}] \label{prop:bulletSOD}
Let $\mcA, \mcB, \mcC, \mcD$ be pre-triangulated dg categories. 
Suppose that we have a semi-orthogonal decomposition 
$H^0(\mcC)=\langle H^0(\mcA), H^0(\mcB) \rangle$. 
Then we have a semi-orthogonal decomposition 
\[
H^0(\mcC \bullet \mcD)
=\langle H^0(\mcA \bullet \mcD), H^0(\mcB \bullet \mcD) \rangle. 
\]
\end{prop}

\subsection{Equivariant categories} 
\label{sec:eqcat}
In this subsection, 
we recall basic facts about equivariant categories. 
We refer to \cite{bo20, ela14} for more details. 

Let $\mcD$ be a $\bC$-linear category, $G$ be a finite group. 
Recall that a {\it $G$-action} on $\mcD$ consists of the following data: 
\begin{itemize}
\item An autoequivalence $\rho_g$ of $\mcD$ for each $g \in G$, 
\item A natural isomorphism 
$\theta_{g, h} \colon \rho_g\rho_h \xrightarrow{\sim} \rho_{gh}$ 
for each pair $g, h \in G$
such that the following diagram commutes for all $g, h, k \in G$: 
\[
\xymatrix{
&\rho_g\rho_h\rho_k \ar[r]^{\rho_g\theta_{h, k}} \ar[d]_{\theta_{g, h}\rho_k}
&\rho_g\rho_{hk} \ar[d]^{\theta_{g, hk}} \\
&\rho_{gh}\rho_k \ar[r]_{\theta_{gh, k}}
&\rho_{ghk}. 
}
\]
\end{itemize}

Given a $G$-action on the category $\mcD$, 
we define the {\it equivariant category} $\mcD^G$ as follows: 
\begin{itemize}
\item An object of $\mcD^G$ is a data $(E, \phi_g)$, 
where $E$ is an object of $\mcD$ and 
$\phi_g \colon E \xrightarrow{\sim} \rho_gE$ is an isomorphism for each $g \in G$, 
such that the following diagram commutes for each pair of elements $g, h \in G$: 
\[
\xymatrix{
&E \ar[r]^{\phi_{gh}} \ar[d]_{\phi_g}
&\rho_{gh}E \\
&\rho_{g}E \ar[r]_{\rho_g\phi_h}
&\rho_g\rho_hE \ar[u]_{\theta_{g, h}}. 
}
\]

\item A morphism between objects $(E, \phi_g)$ and $(F, \psi_g)$ is a morphism 
$f \colon E \to F$ in the category $\mcD$ such that 
the following diagram is commutative for every $g \in G$: 
\[
\xymatrix{
&E \ar[r]^{f} \ar[d]_{\phi_g}
&F \ar[d]^{\psi_g} \\
&\rho_gE \ar[r]_{\rho_g f}
&\rho_gF. 
}
\]
\end{itemize}

For a subgroup $H \subset G$, we have the {\it restriction} and the {\it induction} functors: 
\begin{align*}
&\Res^H_G \colon \mcD^G \to \mcD^H, \quad
\Ind^G_H \colon \mcD^H \to \mcD^G. 
\end{align*}
The restriction functor is defined in an obvious way. 
The induction functor is defined by 
$\Ind^G_H(E, \phi_h) \coloneqq (\oplus_{[g_j] \in G/H} \rho_{g_j}E, \epsilon_g)$, 
where for $g \in G$, the isomoprphism $\epsilon_\sigma$ restricted to the summand 
$\rho_{g_j}E$ is the composition 
\[
\rho_{g_j}E \xrightarrow{\rho_{g_j}\phi_h}
\rho_{g_j}\rho_hE \xrightarrow{\theta_{g_j, h}} 
\rho_{g_jh}E=\rho_{gg_k}E 
\xrightarrow{\theta^{-1}_{g, g_k}} 
\rho_g\rho_{g_k}E. 
\]
Here, the elements $h \in H$ and $g_k \in G$ are defined by 
$g_jh=gg_k$.

Suppose now that $\mcD$ is a dg-enhanced triangulated category 
with a $G$-action. 
It is known that the equivariant category $\mcD^G$ 
is again triangulated (cf. \cite[Corollary 6.10]{ela14}). 

The following Elagin's theorem is crucial for our purpose: 
\begin{thm}[\cite{ela12}, {\cite[Theorem 6.2]{shi18}}] \label{thm:Elagin}
Let $\mcD$ be a dg-enhanced triangulated category, $G$ be a finite group acting on $\mcD$ such that the characteristic of $\bK$ 
does not divide the order of $G$. 
Suppose that we have a semi-orthogonal decomposition 
$\mcD=\langle \mcA, \mcB \rangle$ 
whose components are preserved by the $G$-action. 

Then we have a semi-orthogonal decomposition 
\[
\mcD^G=\langle \mcA^G, \mcB^G \rangle. 
\]
\end{thm}

We also use the following lemma: 
\begin{lem} \label{lem:symAB}
Let $\mcD$ be a triangulated category, 
$G$ a finite group acting on $\mcD$, 
and $H \subset G$ a subgroup. 
Then there exists an equivalence 
\[
\Phi \colon 
\Big(\bigoplus_{[g_j] \in G/H} 
g_j \cdot \mcD
\Big)^{G} 
\xrightarrow{\sim} \mcD^H. 
\]
\end{lem}
\begin{proof}
To define the functor $\Phi$, let us first consider the restriction functor 
\begin{align*}
\Res^{H}_{G} \colon 
\Big(\bigoplus_{[g_j] \in G/H}
g_j \cdot \mcD 
\Big)^{G} 
\to \Big(\bigoplus_{[g_j] \in G/H}
g_j \cdot \mcD
\Big)^{H}. 
\end{align*}
Note that the action of $H$ on 
$\bigoplus_{[g_j] \in G/H} g \cdot \mcD$ preserves 
the direct summand $e \cdot \mcD=\mcD$. 
Hence, by Theorem \ref{thm:Elagin}, 
we have a direct sum decomposition 
\[
\Big(\bigoplus_{[g_j] \in G/H}
g_j \cdot \mcD
\Big)^{H}
=\mcD^H \oplus \mcC 
\]
for some category $\mcC$ . 
Hence we have the projection functor 
\[
\pr \colon 
\Big(\bigoplus_{[g_j] \in G/H}
g_j \cdot \mcD \Big)^{H}
\to \mcD^H. 
\]

Then we define the functor $\Phi$ to be the composition of the functors 
$\pr$ and $\Res^{H}_{G}$. 
It has a left and right adjoint functor 
$\Psi \coloneqq \Ind^{G}_{H} \circ \iota$, 
where 
\[
\iota \colon
\mcD^H \to 
\Big(\bigoplus_{[g_j] \in G/H}
g_j \cdot \mcD \Big)^{H}
\]
is the inclusion as a direct summand. 

We claim that $\Phi$ and $\Psi$ are inverse to each other. 
From the construction, it is obvious that 
\[
\Phi \circ \Psi \cong \id \colon 
\mcD^H \to \mcD^H. 
\]

For the converse, 
pick an element 
\[
x=(E_j, \phi_g) \in 
\Big(\bigoplus_{[g_j] \in G/H}
g_j \cdot \mcD \Big)^{G}, 
\]
where $E_j \in g_j \cdot \mcD$ 
and $\phi_g \colon \oplus_j E_j \xrightarrow{\sim} \rho_{g}(\oplus_j E_j)$ 
are isomorphisms for $g \in G$. 
Then we have 
\[
\Psi \circ \Phi(x)=\Big(
\bigoplus_{[g_j] \in G/H}\rho_{g_j}E_1, 
\epsilon_g \Big), 
\]
where we put $g_1=e$ so that $E_1 \in e \cdot \mcD=\mcD$. 
Recall that, for $g \in G$, the morphism $\epsilon_g$ 
restricted to the summand $\rho_{g_j}E_1$ is defined by the composition 
\begin{equation} \label{eq:def-epsilon}
\rho_{g_j}E_1 \xrightarrow{\rho_{g_j \phi_h}} 
\rho_{g_j}\rho_{h}E_1 \xrightarrow{\theta_{g_j, h}} 
\rho_{g_j h}E_1=\rho_{gg_k}E_1 
\xrightarrow{\theta_{g, g_k}^{-1}} 
\rho_{g}\rho_{g_k}E_1,
\end{equation}
where the elements $g_k \in G$ and $h \in H$ are defined by 
$g_jh=gg_k$. 
We claim that the isomorphism 
\[
\bigoplus_j \phi_{g_j} \colon 
\bigoplus_j E_j \xrightarrow{\sim} 
\bigoplus_j \rho_{g_j}E_1 
\]
commutes with $\phi_g$ and $\epsilon_g$, and hence 
$x$ and $\Psi \circ \Phi(x)$ are isomorphic. 
Namely, we shall show that the following diagram commutes: 
\begin{equation} \label{eq:psiphi=id}
\xymatrix{
&\oplus_jE_j \ar[rr]^{\oplus_j\phi_{g_j}} \ar[d]_{\phi_g} 
& &\oplus_j \rho_{g_j}E_1 \ar[d]^{\epsilon_g} \\
&\rho_g(\oplus_jE_j) \ar[rr]_{\rho_g(\oplus_j\phi_{g_j})}
& &\rho_g(\oplus_j \rho_{g_j}E_1). 
}
\end{equation}
It is enough to show it for each direct summand. 
Fix $[g_{j_0}] \in G/H$ and $g \in G$. 
As above, we let $h \in H$ and $g_{k_0}$ 
denote the elements satisfying $g_{j_0}h=gg_{k_0}$. 
Then the restriction of the diagram (\ref{eq:psiphi=id}) 
to  the direct summand $g_{j_0} \cdot \mcD$ becomes 
\begin{equation} \label{eq:sigma-act}
\xymatrix{
&E_{j_0} \ar[r]^{\phi_{g_{j_0}}} \ar[d]_{\phi_g}
&\rho_{g_{j_0}}E_1 \ar[d]^{\epsilon_{g}} \\
&\rho_g E_{k_0} \ar[r]_{\rho_{g}\phi_{g_{k_0}}} 
&\rho_g\rho_{g_{k_0}}E_1. 
}
\end{equation}
To prove the commutativity of the diagram (\ref{eq:sigma-act}), consider the following diagram: 
\[
\xymatrix{
&
& 
&E_{j_0} \ar@{=}[r] \ar[d]^{\phi_{g_{j_0}h}} \ar[lld]_{\phi_{g_{j_0}}}
&E_{j_0} \ar[r]^{\phi_g} \ar[d]^{\phi_{gg_{k_0}}} 
&\rho_\sigma E_{k_0} \ar[d]^{\rho_g\phi_{g_{k_0}}} \\
&\rho_{g_{j_0}}E_1 \ar[r]_{\rho_{g_{j_0}}\phi_h} 
&\rho_{g_{j_0}}\rho_h E_1 \ar[r]_{\theta_{g_{j_0}, h}} 
&\rho_{g_{j_0}h}E_1 \ar@{=}[r]
&\rho_{gg_{k_0}}E_1 \ar[r]_{\theta_{g, g_{k_0}}^{-1}}
&\rho_g\rho_{g_{k_0}}E_1 .
}
\]
The middle square is commutative as $g_{j_0}h=gg_{k_0}$; 
the left triangle and the right square are commutative by the definition of equivariant objects. 
Hence the whole diagram commutes. 
Moreover, the bottom composition is exactly $\epsilon_g$ 
(see (\ref{eq:def-epsilon})). 
We conclude that the diagram (\ref{eq:sigma-act}) is commutative, as claimed. 
\end{proof}

\subsection{Group actions on dg categories}
Here, we recall some useful facts when a group acts on a dg category. 

\begin{defin}[{\cite[Definition 8.1]{ela14}}]
    Let $\mcD$ be a dg category and $G$ a finite group. 
    A {\it {(dg) $G$-action on $\mcD$}} consists of the following data: 
    \begin{itemize}
\item A dg autoequivalence $\rho_g$ of $\mcD$ for each $g \in G$, 
\item A closed isomorphism 
$\theta_{g, h} \colon \rho_g\rho_h \xrightarrow{\sim} \rho_{gh}$ 
of degree $0$ for each pair $g, h \in G$
such that the following diagram commutes for all $g, h, k \in G$: 
\[
\xymatrix{
&\rho_g\rho_h\rho_k \ar[r]^{\rho_g\theta_{h, k}} \ar[d]_{\theta_{g, h}\rho_k}
&\rho_g\rho_{hk} \ar[d]^{\theta_{g, hk}} \\
&\rho_{gh}\rho_k \ar[r]_{\theta_{gh, k}}
&\rho_{ghk}. 
}
\]
\end{itemize}
\end{defin}

Given a dg caction of a finite group $G$ on a dg category $\mcD$, 
we can define the dg category $\mcD^G$ of $G$-equivariant objects as in Section \ref{sec:eqcat}, 
but the isomorphisms $\phi_g$ are required to be closed and of degree $0$ (see \cite[Definition 8.2, Proposition 8.3]{ela14}). 

We will use the following lemma: 
\begin{lem}[{\cite[Lemma 8.6]{ela14}}]
\label{lem:H0perfG}
    Let $\mcD$ be a dg category, $G$ a finite group acting on $\mcD$ such that the characteristic of $\bK$ 
does not divide the order of $G$. Then we have a natural equivalence
    \[
    H^0\left(\Perf\left(\mcD^G \right) \right) \cong 
    \left(H^0(\Perf(\mcD)) \right)^G. 
    \]
\end{lem}

\section{Proof of the main theorem} \label{sec:main}

\subsection{Notation for symmetric products}
Let $\mcD$ be a dg category, $n >0$ a positive integer. 
Following \cite{gk14}, we define the {\it $n$-th symmetric product} $\Sym^n\mcD$ of $\mcD$ 
as follows: 
\[
\Sym^n\mcD \coloneqq
\Perf \big(
(\mcD^{\bullet n})^{\fS_n}
\big), 
\]
where the symmetric group $\fS_n$ acts on $\mcD^{\bullet n}$ 
by permutations of the components. 
For $n=0$, we define 
$\Sym^0\mcD \coloneqq \Mod_\bK$, the dg category of dg vector spaces. 

\begin{rmk}
It follows from Lemma \ref{lem:H0perfG} that we have 
\[
H^0(\Sym^n\mcD) \cong 
\big(H^0(\mcD^{\bullet n}) \big)^{\fS_n}.
\]
\end{rmk}

\begin{ex}
\label{ex:symIX}
Let $X$ be a smooth projective variety. Then we have 
\[
H^0(\Sym^nI(X)) \cong D^b([X^{\times n}/\fS_n]), 
\]
see \cite[Example 7.1]{gkl21} and 
\cite[Example 2.2.8(a)]{gk14}. 
\end{ex}

For $i=0, \cdots, n$, we fix the embedding 
$\fS_{n-i} \times \fS_i \hookrightarrow \fS_n$, 
where $(\tau, \eta) \in \fS_{n-i} \times \fS_i$ acts on $\{1, \cdots, n\}$ by 
\[
(\tau, \eta) \cdot \{1, \cdots, n\}:=\{\tau(\{1, \cdots, n-i\}), \eta(\{n-i+1, \cdots, n\})\}. 
\]

The quotient $\fS_n/(\fS_{n-i} \times \fS_i)$ has ${n \choose i}$ elements. 
We fix the representatives of cosets in $\fS_n/(\fS_{n-i} \times \fS_i)$: 
\[
\sigma_j \in \fS_n, \quad 
1 \leq j \leq {n \choose i}. 
\]

\subsection{Semi-orthogonal decompositions for symmetric products}
\begin{lem} \label{lem:SODprod}
Let $\mcA, \mcB, \mcD$ be  pre-triangulated dg categories. 
Suppose that we have a semi-orthogonal decomposition 
$H^0(\mcD)=\langle H^0(\mcA), H^0(\mcB) \rangle$. 

Then we have a semi-orthogonal decomposition 
\begin{equation} \label{eq:SODprod}
H^0(\mcD^{\bullet n})=
\Big\langle
\bigoplus_{j=1}^{{n \choose i}}
\sigma_j \cdot H^0(\mcA^{\bullet (n-i)} \bullet \mcB^{\bullet i}) \colon 
i=0, \cdots, n
\Big\rangle. 
\end{equation}
for each positive integer $n >0$. 
\end{lem}
\begin{proof}
Fix an integer $i \in \{0, \cdots, n\}$. 
We first prove that the subcategories 
\[
\sigma_j \cdot H^0(\mcA^{\bullet (n-i)} \bullet \mcB^{\bullet i}) \subset H^0(\mcD^{\bullet n}), 
\quad 1 \leq j \leq {n \choose i}, 
\]
are completely orthogonal to each other. 
Since a perfect module is semi-free, it is enough to show that 
\[
\sigma_j \cdot H^0(\mcA^{\otimes (n-i)} \otimes \mcB^{\otimes i}) \subset H^0(\mcD^{\otimes n}), 
\quad 1 \leq j \leq {n \choose i}, 
\]
are completely orthogonal to each other. 
For each 
$1 \leq j, k \leq {n \choose i}$ with 
$j \neq k$, there exists $l \in \{1, \cdots, n\}$ such that the $l$th component of the tensor product 
$\sigma_j \cdot (\mcA^{\otimes (n-i)} \otimes \mcB^{\otimes i})$ 
(resp. $\sigma_k \cdot (\mcA^{\otimes (n-i)} \otimes \mcB^{\otimes i})$) is 
$\mcB$ (resp. $\mcA$). 
Since $\mcB$ is left-orthogonal to $\mcA$, 
$\sigma_j \cdot (\mcA^{\otimes (n-i)} \otimes \mcB^{\otimes i})$ is left-orthogonal to 
$\sigma_k \cdot (\mcA^{\otimes (n-i)} \otimes \mcB^{\otimes i})$. 
Since $j$ and $k$ are arbitrary as long as $j \neq k$, 
we conclude that 
\[
\sigma_j \cdot H^0(\mcA^{\otimes (n-i)} \otimes \mcB^{\otimes i}) \subset H^0(\mcD^{\otimes n}), 
\quad 1 \leq j \leq {n \choose i}, 
\]
are completely orthogonal to each other, as claimed. 

Now the fact that 
the right hand side of (\ref{eq:SODprod}) 
forms a semi-orthogonal decomposition 
follows from Proposition \ref{prop:bulletSOD}. 
\end{proof}

\begin{thm} \label{thm:SymSOD}
Let $n >0$ be a positive integer. 
Assume that the characteristic of $\bK$ is either zero 
or greater than $n$. 
Let $\mcA, \mcB, \mcD$ be pre-triangulated dg categories. 
Suppose that we have a semi-orthogonal decomposition 
$H^0(\mcD)=\langle H^0(\mcA), H^0(\mcB) \rangle$. 

Then we have a semi-orthogonal decomposition 
\[
H^0(\Sym^n(\mcD))=\left\langle
H^0(\Sym^{n-i}\mcA \bullet \Sym^i\mcB) \colon 
i=0, \cdots, n
\right\rangle. 
\]
\end{thm}
\begin{proof}
First, note that the semi-orthogonal component 
\[
\bigoplus_{j=1}^{{n \choose i}}
\sigma_j \cdot H^0(\mcA^{\bullet (n-i)} \bullet \mcB^{\bullet i})
\subset H^0(\mcD^{\bullet n})
\]
in Lemma \ref{lem:SODprod} is preserved by the action of $\fS_n$ 
for each $i=0, \cdots, n$. 
Hence by Elagin's Theorem \ref{thm:Elagin}, 
the semi-orthogonal decomposition (\ref{eq:SODprod}) 
descends to a semi-orthogonal decomposition on 
the category $H^0(\Sym^n\mcD) \cong 
(H^0(\mcD^{\bullet n}))^{\fS_n}$ with semi-orthogonal summands 
\[
\Big(
\bigoplus_{j=1}^{{n \choose i}}
\sigma_j \cdot H^0(\mcA^{\bullet (n-i)} \bullet \mcB^{\bullet i})
\Big)^{\fS_n}. 
\]
By Lemma \ref{lem:symAB}, 
we have an equivalence 
\[
\Big(
\bigoplus_{j=1}^{{n \choose i}}
\sigma_j \cdot H^0(\mcA^{\bullet (n-i)} \bullet \mcB^{\bullet i})
\Big)^{\fS_n} 
\cong 
\Big(
H^0(\mcA^{\bullet (n-i)} \bullet \mcB^{\bullet i})
\Big)^{\fS_{n-i} \times \fS_i}. 
\]
It remains to show that the right hand side of the above equivalence is equivalent to 
$H^0(\Sym^{n-i}\mcA \bullet \Sym^i\mcB)$. 
By Lemma \ref{lem:H0perfG}, we have 
\begin{align*}
\Big(H^0(\mcA^{\bullet (n-i)} \bullet \mcB^{\bullet i})
\Big)^{\fS_{n-i} \times \fS_i} 
&\cong 
H^0\Big(
\Perf\Big(\big(
\mcA^{\bullet (n-i)} \bullet \mcB^{\bullet i}
\big)^{\fS_{n-i} \times \fS_i}
\Big)
\Big) \\
&\cong H^0\big(\Sym^{n-i}\mcA \bullet \Sym^i\mcB\big), 
\end{align*}
as required. 
\end{proof}

\section{Geometric examples} \label{sec:Hilb}
For simplicity, we assume that the base field $\bK$ is algebraically closed of characteristic zero. 
In this section, we consider the case when the dg-category $\mcD$ 
is the standard enhancement $I(X)$ 
of $D^b(X)$, where $X$ is a smooth projective variety. 
As we saw in Example \ref{ex:symIX}, 
we have an equivalence 
$H^0(\Sym^nI(X)) \cong D^b([X^{\times n}/\fS_n])$ 
for each positive integer $n >0$.

\subsection{$\dim X=0$}
In this case, we have 
(without using Theorem \ref{thm:SymSOD}) 
\begin{equation}\label{eq:sympt}
D^b([\pt/\fS_n])
=\left\langle p(n)\mbox{-copies of } D^b(\pt) 
\right\rangle,
\end{equation}

where $p(n)$ is the number of partitions of $n$, 
which coincides with the number of irreducible representations of $\fS_n$. 
Note also that this semi-orthogonal decomposition is completely orthogonal. 

\subsection{$\dim X=1$}
Let $X=C$ be a smooth projective curve. 
If $g(C) \geq 1$, 
then the derived category $D^b(C)$ 
does not admit a non-trivial 
semi-orthogonal decomposition 
by \cite{oka11}. 
When $C=\bP^1$, 
we have 
$D^b(\bP^1)
=\langle D^b(\pt), D^b(\pt) \rangle$, 
and hence 
Theorem \ref{thm:SymSOD} implies 
\begin{equation} \label{eq:symP1}
\begin{aligned}
D^b([(\bP^1)^n/\fS_n])&=
\left\langle
D^b\left([\pt/\fS_{n-i}] \times [\pt/\fS_i]\right) 
\colon i=0, \cdots, n
\right\rangle \\
&=\left\langle
\sum_{i=0}^n p(n-i) \cdot p(i) \mbox{-copies of } D^b(\pt)
\right\rangle. 
\end{aligned}
\end{equation}

On the other hand, we have the following result in all genera: 
\begin{thm}[{\cite[Theorem B]{pvdb19}}] \label{thm:PVdB}
Let $C$ be a smooth projective curve, $n>0$ a positive integer. 
Then we have a semi-orthogonal decomposition 
\[
D^b([C^{\times n}/\fS_n])=
\left\langle
D^b\left(
\prod_{a_i} \Sym^{a_i}C
\right)
\colon a_i \in \bZ_{\geq 0}, \sum_{i=1}^n ia_i=n
\right\rangle. 
\]
\end{thm}

\subsection{$\dim X=2$}
Let $X=S$ be a smooth projective surface. 
In this case, we have 
$D^b([S^{\times n}/\fS_n]) \cong D^b(\Hilb^n(S))$ 
by the derived McKay correspondence \cite{bkr01, hai01}.

\subsubsection{Surfaces with full exceptional collections}
\begin{cor}
\label{cor:fullexcep}
Let $\mcD$ be a dg-enhanced triangulated category. Suppose that $\mcD$ has a full exceptional collection of length $l$. 
Then $\Sym^n\mcD$ has a full exceptional collection of length 
\begin{equation} \label{eq:symexcep}
q(n; l) \coloneqq \sum_{\substack{i_1+\cdots+i_l=n \\ i_1, \cdots, i_l \geq 0}}
p(i_1) \cdot \cdots \cdot p(i_l). 
\end{equation}
\end{cor}
\begin{proof}
We prove it by induction on $l$. 
When $l=1$, the result follows from (\ref{eq:sympt}). 

Fix an integer $l \geq 2$ and 
assume that the assertion holds 
for $l-1$ and any $n \geq 1$. 
Suppose that $\mcD$ has a full exceptional collection of length $l$, and consider the 
following semi-orthogonal decomposition: 
\[
\mcD=\langle 
D^b(\pt), \mcA
\rangle, 
\]
where $\mcA$ has a full exceptional collection of length $l-1$. 
By Theorem \ref{thm:SymSOD}, we have 
a semi-orthogonal decomposition 
\[
\Sym^n\mcD=\langle
\Sym^{n-i}D^b(\pt) \bullet \Sym^i\mcA 
\colon 0 \leq i \leq n
\rangle. 
\]
By the induction hypothesis, 
$\Sym^i\mcA$ has a full exceptional collection of length $q(i; l-1)$ for each $0 \leq i \leq n$. 
Hence $\Sym^n\mcD$ has a full exceptional collection of length 
\[
\sum_{i=0}^n p(n-i)q(i;l-1)
=q(n; l), 
\]
as claimed. 
\end{proof}

In particular, we have: 
\begin{cor}
Let $S$ be a smooth projective surface. 
Suppose that $D^b(S)$ has a full exceptional collection of length $l$. 
Then $D^b(\Hilb^n(S))$ has a full exceptional collection of length 
$q(n; l)$. 
\end{cor}

The above corollary recovers the result of Elagin \cite[Theorem 2.3]{ela09}, 
see also \cite[Proposition 1.3]{ks15}. 
The construction generalizes to the case when an exceptional collection is not full. 
The case of Enriques surfaces is treated in \cite{ks15}. 
We consider another interesting case of fake projective planes in the next subsection.

\subsubsection{Quasi-phantoms}
Let $l>0$ be a positive integer. 
Suppose that $D^b(S)$ has the following semi-orthogonal decomposition: 
\begin{equation}\label{eq:Sphantom}
 D^b(S)=\langle
(l+2) \mbox{-copies of } D^b(\pt), \mcA
\rangle    
\end{equation}
with $\HH_*(\mcA)=0$, 
where $\HH_*(-)$ denotes the Hochschild homology. 
Such a category $\mcA$ is called a 
{\it quasi-phantom} category. 
See \cite{ao13, bgks15, bgs13, kkl17, gkms15, gs13} and their references for examples of quasi-phantom subcategories inside the derived category of a smooth projective surface. 

In this case, we have 
\[
\HH_i(D^b(S))=\begin{cases}
    \bK^{\oplus l+2} & (i=0) \\
    0 & (i \neq 0). 
\end{cases}
\]
On the other hand, by the HKR isomorphsim, we have 
\[
\HH_i(D^b(S)) \cong 
\bigoplus_{p-q=i-2} H^p(S, \Omega^q_S), 
\]
and hence the Betti numbers of $S$ are determined as follows: 
\[
b_i(S)=\begin{cases}
    1 & (i=0, 4) \\
    l & (i=2) \\
    0 & (\mbox{otherwise}). 
\end{cases}
\]

\begin{lem}\label{lem:Hilbphantom}
Let $S$ be a smooth projective surface. 
Suppose that $D^b(S)$ has a semi-orthogonal decomposition (\ref{eq:Sphantom}), 
where $\mcA$ is a quasi-phantom category. 
Then $\Sym^i(\mcA)$ is a quasi-phantom category for 
any $i \geq 1$. 
\end{lem}
\begin{proof}
By Theorem \ref{thm:SymSOD}, we have 
\[
D^b(\Hilb^n(S))=\left\langle
q(n-i; l+2) \mbox{-copies of } \Sym^i\mcA 
\colon i=0, \cdots, n
\right\rangle 
\]
for any positive integer $n>0$. 

Let $T$ denote the blow-up of $\bP^2$ 
at $(l-1)$ points. 
Then $D^b(T)$ has a full exceptional collection of length $l+2$, 
and we have $b_i(S)=b_i(T)$ for all $i$. 
By G{\"o}ttsche's formula \cite{got90}, 
the Betti numbers of $\Hilb^n(S)$ are determined by that of the surface $S$. 
In particular, we have $b_i(\Hilb^n(S))=b_i(\Hilb^n(T))$ 
for all $n >0$ and $i \geq 0$. 
Hence by the HKR isomorphism, we have 
\begin{align*}
q(n; l+2)=\dim \HH_*(\Hilb^n(T))&=\sum_i b_i(\Hilb^n(T)) \\
&=\sum_i b_i(\Hilb^n(S)) =\dim \HH_*(\Hilb^n(S)). 
\end{align*}
Noting that $\Sym^0\mcA=D^b(\pt)$, 
we conclude that 
$\Sym^i\mcA$ are phantom subcategories for all $i \geq 1$
\end{proof}

\subsubsection{Ruled surfaces} \label{sec:ruled}
Let $S \to C$ be a $\bP^1$-bundle over a curve $C$. 
Then we have a semi-orthogonal decomposition 
$D^b(S)=\langle D^b(C), D^b(C) \rangle$. 
Hence Theorems \ref{thm:SymSOD} and \ref{thm:PVdB} imply 
\begin{align*}
&\quad D^b(\Hilb^n(S)) \\
&=\left\langle
D^b\left([C^{\times (n-i)}/\fS_{n-i}] \times [C^{\times i}/\fS_i] \right) 
\colon i=0, \cdots n
\right\rangle \\
&=\left\langle
D^b\left( 
\prod_{a_j, b_k}\Sym^{a_j}C \times \Sym^{b_k}C
\right)
\colon
\begin{aligned}
&\sum_j ja_j=n-i,  \sum_kkb_k=i, \\
&i=0, \cdots, n, \quad a_j, b_k \in \bZ_{\geq 0}
\end{aligned}
\right\rangle. 
\end{align*}
Compare it with the motivic formula in \cite[Example 4.9]{got01}. 

Moreover, by \cite[Corollary 5.12]{tod21d}, the derived categories $D^b(\Sym^NC)$ 
for $N \geq g(C)$ 
further decompose into the derived categories of the Jacobian $J(C)$ 
and $\Sym^iC$ for $0 \leq i \leq g(C)-1$. 
We conclude that 
$D^b(\Hilb^n(S))$ has a semi-orthogonal decomposition 
whose components are derived categories of the products of 
$J(C)$ and $\Sym^iC$ for $0 \leq i \leq \min\{n, g(C)-1\}$.

\subsubsection{Blow-ups} \label{sec:blup}
Let $\hatS \to S$ be the blow-up at a point. 
We have a semi-orthogonal decomposition 
$D^b(\hatS)=\langle D^b(S), D^b(\pt) \rangle$. 
By Theorem \ref{thm:SymSOD}, we have 
\[
D^b(\Hilb^n(\hatS))=\left\langle
p(i) \mbox{-copies of } D^b(\Hilb^{n-i}(S)) \colon i=0, \cdots, n
\right\rangle. 
\]

This categorifies the blow-up formula for the Euler characteristics of the Hilbert schemes. 
Note that the same semi-orthogonal decomposition 
was obtained in the author's previous paper \cite[Theorem 1.1 (2)]{kos21b} 
via a completely different method.

\end{document}